\documentclass[10pt]{amsart}
\usepackage{fontenc}
\usepackage[english]{babel}
\usepackage{mathtools}
\usepackage[swapnames,normal]{frontespizio}

\theoremstyle{plain}
\newtheorem{theorem}{Theorem}[section]
\newtheorem{lemma}[theorem]{Lemma}
\newtheorem{proposition}[theorem]{Proposition}
\newtheorem{corollary}[theorem]{Corollary}

\theoremstyle{remark}
\newtheorem{example}[theorem]{Example}
\theoremstyle{remark}
\newtheorem{remark}[theorem]{Remark}
\theoremstyle{definition}
\newtheorem{definition}[theorem]{Definition}
\numberwithin{equation}{section}
\def\d{\text{dist}}
\def\D{\mathcal{V}}
\def\<{\langle}
\def\>{\rangle}
\def\F{\mathcal{F}}
\def\1{\mathbf{1}}
\def\R{\mathbb{R}}
\def\S{\mathcal{S}}
\def\H{\mathfrak{H}}
\def\E{\mathbb{E}}
\def\P{\mathbb{P}}
\def\A{\mathfrak{A}}
\def\N{\mathbb{N}}
\def\B{\mathcal{B}}
\def\s{\text{Sym}_d(\R)}
\def\C{\mathbb{C}}
\def\f{\mathfrak{F}}
\def\r{\mathbb{R}\times i\mathbb{R}^d\times \s}

\begin{document}

\title{A quasilinear approach to fully nonlinear parabolic (S)PDE\MakeLowercase{s} on $\mathbb{R}^d$}
\author{Antonio Agresti}

\address{Dipartimento di Matematica Guido Castelnuovo\\
Sapienza Universit\`a di Roma\\
P.le A. Moro 2, 00100 Roma, Italy}

\begin{abstract}
We study the Cauchy problem for fully nonlinear (stochastic) parabolic partial differential equations. We provide both in deterministic and stochastic case the existence of a maximal defined solution for the problem and we provide suitable blow-up criterion.\\
The key idea is the use of the paradifferential operator calculus in order to reduce the fully nonlinear problem into an abstract quasilinear (stochastic) parabolic equation. This allows us to use some recent results on abstract quasilinear (stochastic) evolution equations in Banach spaces. To do so, we analyse the properties of the paradifferential operator, in light of known results on the boundedness of the $\mathcal{H}^{\infty}$-calculus for pseudodifferential operator.\\
Finally, we extend the theory just developed to cover high order fully nonlinear parabolic (S)PDEs.
\end{abstract}
\maketitle
\tableofcontents
\section{Introduction}
\label{s:introduction}
In this article we provide a very general framework to deal with fully nonlinear (stochastic) parabolic partial differential equations on $\mathbb{R}^d$, but the method can be extended to more general situations (see Section \ref{s:extensions} for more on this).\\ 
\\
Fully nonlinear (stochastic) PDEs are not as studied as semilinear and quasilinear (stochastic) PDEs. In the deterministic case some methods has been developed (see \cite{DaPrato-maximal,Lunardi}) but in the stochastic case not many result on existence are available, see \cite{Stochastic,DaPrato,P.L.Lions,P.L.Lions2} for some results.\\
The narrowness of the results in stochastic case seems to be due to the fact that the semigroup (or more generally abstract) approaches are not available at first sight. For deterministic parabolic PDEs, in an abstract setting some results are proved by linearization technique in \cite{DaPrato-maximal,Lunardi} and the results in \cite{DaPrato} are indeed related to this one.\\
Our method is completely different from the above mentioned ones. Indeed, we are able to recast the fully nonlinear problem in an abstract quasilinear (stochastic) evolution equation for which many sharp existence results are now available, this is due to the recent progresses in the theory of (stochastic) maximal regularity; see \cite{QuasilinearStochastic,Pruss,StochasticMaximal,StochasticEvolution,Weis}.\\
To do this, we take advantage of paradifferential technique, which has been shown quite useful in study of fully nonlinear elliptic and parabolic equations; see \cite{Taylor,Tools}. Besides the abstract quasilinear formulation of the fully nonlinear problem, our main task is to prove that the operator thus defined has good properties as the ordinary differential operator, which are necessary to use the quasilinear abstract existence results provided in \cite{QuasilinearStochastic,Pruss}.\\
The novelty of this approach is the possibility to work with the same basic idea both in deterministic and stochastic settings. Furthermore, this approach makes available a lot of known results on abstract quasilinear evolution equation which are at first sight not intended for fully nonlinear problems.\\
We refer to Section \ref{s:comments} for further comments and comparison with the methods already available.\\
\\
Now, we turn to explain in details the strategy of the proofs, which will be worked out in the subsequent sections. 
To begin, we introduce some notations: 
\begin{itemize}
\item Multi index notation, $\alpha\in \mathbb{N}_0^d$, we set 
$$D^{\alpha}=\frac{1}{i^{|\alpha|}}\frac{\partial^{|\alpha|}}{\partial x^{\alpha_1}_1\dots x^{\alpha_n}_n}\,,$$
where $|\alpha|=\sum_{i=1}^d \alpha_i$. Furthermore we set $\xi^{\alpha}=\xi_1^{\alpha_1}\dots \xi_d^{\alpha_d}$.
\item
The map 
\begin{equation*}
F:\mathbb{R}^d \times \mathbb{R}\times i\mathbb{R}^d\times \s\longrightarrow \mathbb{R}\,,
\end{equation*}
will be denoted by $F(x,\zeta)$, where $\zeta=(\zeta_{\alpha}\,:\,|\alpha|\leq 2) \in \mathbb{R}\times i\mathbb{R}^d\times \s$ and $\s$ is the set of all symmetric matrices of order $d$.
\item Let $l$ be an integer, for each $x\in \R^d$ we define $D^l u(x)\in \r$ as 
\begin{equation*}
(D^l u(x))_ {\alpha} =D^{\alpha}u(x)\,, \qquad |\alpha|\leq 2\,;
\end{equation*}
here $D^0 u =u$.
\end{itemize}
In the first part of this section is useful to introduce the following hypotheses:
\begin{itemize}
\item[(H1)] $s>0$, $q\in (1,+\infty)$ and $q>(2+d)/s$.
\item[(H2)] The initial value $u_0$ belongs to $W^{2+s-2/q}_{q}(\R^d)$, for $s,q$ as in (H2).
\item[(H3)] $F$ is a smooth function of its arguments, its derivatives are uniformly bounded in $x$ and $F(x,0)\in H^s_q(\R^d)$ ($s,q$ are as in (H2)).
\item[(H4)](Strongly parabolicity hypothesis). For each $n\in \mathbb{N}$ there exists $c'_n>0$ such that
\begin{equation*}
\sum_{|\alpha|=2}\frac{\partial F}{\partial \zeta_{\alpha}}(x,\zeta)\, \xi^{\alpha} \geq c'_n |\xi|^2\,,\qquad \forall \xi
\in \mathbb{R}^d\,,
\end{equation*}
uniformly in $x\in \mathbb{R}^d$ and $|\zeta|\leq n$.
\end{itemize}
Here, for any $r\in (1,\infty)$ and $m\in \R$, $H^m_r(\R^d)$ and $W^m_r(\R^d)$ denote respectively the Bessel-potential spaces and the Sobolev-Slobodeckij spaces; see \cite{Interpolation,Analysis1} for more on this. For notational convenience, in the following we set $H^m_r:=H^m_r(\R^d)$ and $W^m_r:=W^m_r(\R^d)$.\\
The following theorem will be proved in Subsection \ref{ss:deterministic proof}.
\begin{theorem}
\label{t:well posedness}
If the assumptions (H1)-(H4) are satisfied, the following fully nonlinear parabolic problem
\begin{equation}
\label{EQ}
\begin{cases}
\partial_t u(x,t) + F(x,D^2u(x,t))=0\,,\qquad & x\in \mathbb{R}^d\,,\,t>0\,,\\
u(x,0)=u_0(x)\,,\qquad & x\in \mathbb{R}^d\,,
\end{cases}
\end{equation}
has a maximal defined solution $u$ in
\begin{equation}
\label{eq:regularity solution}
H^1_{q,loc}([0,T);H^s_q)\cap L_{q,loc}([0,T);H^{2+s}_q)\cap C([0,T);W^{2+s-\frac{2}{q}}_{q})\,;
\end{equation}
where $T=T(u_0)$. Furthermore, $T(u_0)=+\infty$ or $T(u_0)<+\infty$ and it is characterized by
\begin{equation}
\label{eq:lim extended}
\lim_{t\nearrow T(u_0)} u(t)\,\,\,\text{does not exist in}\,\, W^{2+s-\frac{2}{q}}_{q}\,.
\end{equation} 
\end{theorem}
The main idea for the proof of Theorem \ref{t:well posedness} is to recast the problem (\ref{EQ}) as an abstract quasilinear parabolic equation via paradifferential operator.\\
As explained in Section \ref{s:paradifferential}, for $u\in C^{2+r}_*$ with $r>0$, we may write
\begin{equation}
\label{eq:intro decomposition}
F(x,D^2u)= M_F(u;x,D)u + F(x,D^2\Psi_0(D)u)\,,
\end{equation}
pointwise. Where $\Psi_0\in C^{\infty}_0(\mathbb{R}^d)$, $\Psi_0(D)$ and $M_F(u;x,D)$ are the pseudodifferential operator  respectively with symbol $\Psi_0(\xi)$ and
\begin{equation*}
M_F(u;x,\xi)\in S^m_{1,1}\cap C^r_* S^m_{1,0}\,.
\end{equation*}
Here, $C^r_*:=C^r_*(\mathbb{R}^d):=B^r_{\infty,\infty}(\mathbb{R}^d)$ denotes the Zygmund spaces; for the definition of the symbol classes $S^m_{1,1}$, $C^r_* S^m_{1,0}$ see Definitions \ref{def:pseudodiff}-\ref{def:pseudodiff rough} and \cite{Taylor,Tools} for a comprehensive study.\\
With a slight abuse of notation, we denote by $M_F(u)$ the realization of the paradifferential operator on $H^s_q$ and $G(u):=-F(\cdot,\Psi_0(D)D^2u)$. Then the proof of Theorem \ref{t:well posedness} is equivalent to the solvability of following problem 
\begin{equation}
\label{eq:quasilinear}
\begin{cases}
u' + M_F(u)u = G(u)\,,\,\,t>0\,,\\
 u(0)=u_0\,.
\end{cases}
\end{equation} 
We will show that the problem in (\ref{eq:quasilinear}) can be handled with the theory of abstract quasilinear parabolic differential equations, as developed in Chapter 5 of \cite{Pruss}, but it will require some work.\\
To begin, recall that 
$$(H^s_q,H^{s+2}_q)_{1-1/q,q}=W^{s+2-2/q}_q\,.$$
Furthermore, for any $v\in W^{s+2-2/q}_q$ and $n\in \mathbb{N}$ define
\begin{equation*}
B_{s,q}(v,n):=\{u\in W^{s+2-2/q}_q:\,|u-v|_{W^{s+2-2/q}_q}<n\}\,.
\end{equation*} 
The following proposition plays an essential role in proving existence and a maximal defined solution for the problem in (\ref{eq:quasilinear}) via known results for abstract quasilinear parabolic problem; for the notation we refer to Subsection \ref{ss:Hinfinite and maximal}.
\begin{proposition}
\label{prop:Hinfinite M}
Let the hypotheses (H1), (H3)-(H4) be satisfied and let $n\in \mathbb{N}$ be an integer. Then there exist a $\lambda_n\in \mathbb{R}^+$ such that $\lambda_n+ M_F(u)\in \mathcal{H}^{\infty}(H^s_q)$ with $\phi^{\infty}_{M_{F}(u)}\leq \phi_n<\pi/2$ for any $u\in B_{s,q}(0,n)$.\\
Furthermore, there exists $M_n \in \mathbb{R}^+$, such that
\begin{equation}
\label{eq:prop M uniform of H infinite}
|f(\lambda_n+ M_F(u))|_{\mathcal{B}(H^s_q)} \leq M_n |f|_{H^{\infty}(\Sigma_{\phi_n})}\,,
\end{equation}
for any $f\in H^{\infty}_0(\Sigma_{\phi_n})$ and $u\in B_{s,q}(0,n)$.
\end{proposition}
We will prove Proposition \ref{prop:Hinfinite M} in Subsection \ref{ss:proof of Propositon and P1}.\\
We note that the previous result implies that $\lambda_n+M_{F}(u)$ has the maximal $L_q$-regularity on $H^s_q$ and it allows us to use the abstract theory developed in \cite{Pruss} Chapter 5.\\ 
We point out that, to deal with deterministic fully nonlinear parabolic equations the boundedness of the $\mathcal{H}^{\infty}$-calculus in $H^s_q$ is not necessary; instead it will be of central importance for stochastic parabolic PDEs.\\ 
\\
To give a flavour of the result on fully nonlinear Stochastic PDEs, at this point we consider a fairly simple situation in order to avoid too much technicalities; for a more general situation see Subsection \ref{ss:stochastic proof}.\\
Let $(\Omega,\A,\P)$ be a probability space on which a standard Brownian motion $\{\beta_t\}_{t\in \R^+}$ is defined;  we consider on the probability space the filtration $\f^b:=\{\f_t^b\}_{t\in[0,T]}$ where $\f_t^{b}$ is $\sigma(\beta_s\,:\,s\in [0,t])$ augmented by all $\P$-null sets in $\A$.\\
Let $T>0$ fixed, consider the following problem:
\begin{equation}
\label{eq:stochastic intro}
\begin{cases}
du(t)+ F(x,D^2 u)dt = f(t)dt + B(t,u(t),D^1 u(t))d\beta_t\,,\;\; t\in[0,T],\\
u(0)=u_0\,. 
\end{cases}
\end{equation}
Here:
\begin{itemize}
\item[(H5)] $s>0$, $q\geq 2$ and $q>(2+d)/s$.
\item[(H6)] $u_0:\Omega \rightarrow W^{s+2-2/q}_q$ is strongly $\f_0^b$-measurable.
\item[(H7)] $B:\Omega \times [0,T]\times \R\times \R^d \rightarrow \R$ is measurable and $\omega \mapsto B(\omega,t,y,z)$ is $\f_t^b$ measurable for all $t\in [0,T]$, $y\in \R$ and $z\in \R^d$.\\
Furthermore, $|B(\omega,t,0)|_{L_{\infty}(\Omega; H^s_q)}<\infty$ and
\begin{equation*}
|B(\omega,t,u,D^1 u)-B(\omega,t,v,D^1 v)|_{H^s_q}\leq C|u-v|_{W^{s+2-2/q}_q}\,;
\end{equation*}
for all $u,v\in W^{s+2-2/q}_q$ and $\omega\in \Omega$.
\item[(H8)] $f\in L_q(\Omega\times [0,T];H^s_q)$ is strongly measurable and $\f^b$-adapted.
\end{itemize}
With the same considerations of the deterministic case, the fully nonlinear stochastic problem (\ref{eq:stochastic intro}) can be written as
\begin{equation}
\label{eq:stochastic}
\begin{cases}
du+ M_{F}(u)u \,dt =(G(u)+f)dt+ B(u)dW_{\R}^b\,,\qquad t\in [0,T]\,,\\
 u(0)=u_0\,.
\end{cases}
\end{equation}
See Example \ref{ex:1 dimensional Brownian motion} for $W_{\R}^b$ and Definition \ref{def:solution stochastic} in Subsection \ref{ss:stochastic proof} for the notion of solution to the problem (\ref{eq:stochastic}).\\ 
The following result provides the well posedness of the previous problem, see Theorem \ref{t:wellposedness stochastic refined} for a more general statement. 
\begin{corollary}
\label{t:wellposedness stochastic}
If the assumptions (H3)-(H8) are satisfied. Then, there is a maximal unique local solution  $(u,\{\mu_n\}_n,\mu)$ of the problem (\ref{eq:stochastic}), (see Definition \ref{def:solution stochastic}).
\end{corollary}
See also (\ref{eq:stochasticblow-up criterion}) in Theorem \ref{t:wellposedness stochastic refined} for a blow-up criterion.\\
\\
To complete the proof of Theorems \ref{t:well posedness} and \ref{t:wellposedness stochastic} we will need the following properties of $M_{F}(\cdot)$ and $G(\cdot)$, which complement the content of Proposition \ref{prop:Hinfinite M}.
\begin{itemize}
\item[(P1)] $M_F:W^{s+2-2/q}_q\rightarrow \mathcal{B}(H^{s+2}_q,H^s_q)$ is locally Lipschitz continuous, i.e. for each $n\in \mathbb{N}$ there exists a positive constant $C_n$ such that
\begin{equation*}
|M_F(u)-M_F(v)|_{\mathcal{B}(H^{s+2}_q,H^s_q)} \leq C_n |u-v|_{W^{s+2-2/q}}\,,
\end{equation*}
for any $u,v\in B_{s,q}(0,n)$.
\item[(P2)] $G:W^{s+2-2/q}\rightarrow H^s_q$ is locally Lipschitz continuous, i.e. for each $n\in \mathbb{N}$ there exists a positive constant $C_n$ such that
\begin{equation*}
|G(u)-G(v)|_{H^s_q} \leq C_n |u-v|_{W^{s+2-2/q}}\,,
\end{equation*}
for any $u,v\in B_{s,q}(0,n)$.
\end{itemize}
The paper is organized as follows:\\ 
Section \ref{s:H infinite bounded} is divided into three parts: in Subsection \ref{ss:Hinfinite and maximal} we collect known results on $\mathcal{H}^{\infty}$-calculus and maximal $L_q$-regularity useful in the proof of Theorems \ref{t:well posedness}-\ref{t:wellposedness stochastic}, in Subsection \ref{ss:pseudodifferential} we list some definitions and useful facts about pseudodifferential operators; in the last Subsection \ref{ss:Hinfinite strongly elliptic} we prove that a strongly elliptic operator (see Definition \ref{def:pseudodiff}) admits an $\mathcal{H}^{\infty}$-calculus on $H^s_q$ with $\mathcal{H}^{\infty}$-angle less than $\pi/2$. This result is a simple application of a more general Theorem proved in \cite{Hinfinite} but it will be of central importance in the proof of Proposition \ref{prop:Hinfinite M}.\\
Section \ref{s:paradifferential} contains two parts: in Subsection \ref{ss:Introduction and P2} we briefly explain the construction of paradifferential operator and we prove property (P2), in Subsection \ref{ss:proof of Propositon and P1} we prove the Proposition \ref{prop:Hinfinite M} and property (P1).\\
Section \ref{s:proof of the main results} is divided into two parts, in the first Subsection \ref{ss:deterministic proof} we prove Theorem \ref{t:well posedness} and discuss how to derive further regularity results. Subsection \ref{ss:stochastic proof} is devoted to the Proof of Theorem \ref{t:wellposedness stochastic}. 
The last section is devoted to some comments, extensions and further applications of the approach just explained. 
\section{Preliminaries}
\label{s:H infinite bounded}
\subsection{$\mathcal{H}^{\infty}$-Calculus and Maximal $L_q$-Regularity}
\label{ss:Hinfinite and maximal}
In this subsection we collect some known results and definition which will be useful in the rest of the article. Through this section $X$ denotes a complex Banach space.\\
Let $\phi\in(0,\pi)$ and $\Sigma_{\phi}:=\{z\in \C\setminus \{0\}\,:\,|\arg z|<\phi\}$. We begin with a definition:
\begin{definition}[Sectorial operator] Let $A$ be closed densely defined linear operator on $X$ with dense range. Then $A$ is called a sectorial operator (or briefly $A\in S(X)$) if there exists a $0<\phi<\pi$ such that the resolvent $\rho(A)\supset \C\setminus \Sigma_{\phi}$ and
\begin{equation}
\label{eq:sectorial}
\sup_{z\in  \C\setminus \Sigma_{\phi}} |z (z-A)^{-1}|_{\B(X)}<\infty\,.
\end{equation}
Moreover, the infimum of $\phi\in(0,\pi)$ for which the condition (\ref{eq:sectorial}) holds is called the spectral angle of $A$ and it will be denoted by $\phi_A$.
\end{definition}
Let $\phi\in (0,\pi)$, we denote with $H_0^{\infty}(\Sigma_{\phi})$ the space of the holomorphic function $f$ on the sector $\Sigma_{\phi}$, such that 
\begin{equation*}
\Big\{\sup_{|z|\leq 1\,, z\in \Sigma_{\phi}}|z|^{-\varepsilon} |f(z)| + \sup_{|z|\geq 1\,,z\in \Sigma_{\phi}}|z|^{\varepsilon}|f(z)|\Big\}<\infty\,,
\end{equation*}
for some $\varepsilon>0$.\\
If $A\in S(X)$, let $f\in H^{\infty}_0(\Sigma_\phi)$ for $\phi>\phi_A$ and $\Gamma_{\nu}=\{e^{i\nu}t\,:\,\infty>t\geq 0\} \cup \{e^{-i\nu}t\,:\,0<t<\infty\}$ where $\phi>\nu>\phi_{A}$, then the Dunford integral 
\begin{equation}
\label{eq:f(A)}
f(A):=\frac{1}{2\pi i} \int_{\Gamma_{\nu}} f(\zeta)(\zeta+A)^{-1}\,d\zeta\,,
\end{equation}
is well defined and it converges in $\mathcal{B}(X)$. It can be shown that $f(A)$ does not depend on the value $\nu\in (\phi,\phi_A)$; see \cite{Analysis2,Pruss}.
\begin{definition}[bounded $\mathcal{H}^{\infty}$-calculus]
\label{def:H infinite calculus}
A sectorial operator $A$ on $X$ is said to admit a bounded $\mathcal{H}^{\infty}$-calculus, or briefly $A\in \mathcal{H}^{\infty}(X)$, if there are $\phi>\phi_A$ and a constant $C\in \mathbb{R}^+$ such that
\begin{equation}
\label{eq:H infinite boundness}
|f(A)|_{\mathcal{B}(X)}\leq C|f|_{H^{\infty}(\Sigma_{\phi})}\,, \qquad \forall f\in H_0^{\infty}(\Sigma_{\phi})\,;
\end{equation}
here $|f|_{H^{\infty}(\Sigma_{\phi})}:=\sup_{z\in\Sigma_{\phi}}|f(z)|$ and $f(A)$ is as in (\ref{eq:f(A)}).\\
Furthermore $\phi_{A}^{\infty}$ denotes the $\mathcal{H}^{\infty}$-angle of $A$, i.e. the infimum $\phi>\phi_A$ for which the condition (\ref{eq:H infinite boundness}) holds for some positive constant $C$.
\end{definition}
We conclude this section collecting some facts about maximal regularity. For the notion of mild solution see \cite{Pazy}.
\begin{definition}[Maximal $L_q$-regularity]
\label{def:maximal}
Let $A:D(A)\subset X\rightarrow X$ be a closed densely defined operator on a Banach space $X$. For $q\in(1,\infty)$, we say that $A$ belong to the class $\mathcal{MR}_q(X)$ of maximal $L_q$-regularity if for all $f\in C^1(\mathbb{R}^+,D(A))$ the problem
\begin{equation*}
u'+Au=f\,,\,\,\, t>0\,, \qquad u(0)=0\,,
\end{equation*}
has a unique solution $u$ classical solution, which satisfy
\begin{equation*}
|u'|_{L_q(\mathbb{R}^+,X)}+|Au|_{L_q(\mathbb{R}^+,X)} \leq C |f|_{L_q(\mathbb{R}^+,X)}\,,
\end{equation*}
for some $C>0$ independent on $f$.
\end{definition}
\begin{remark}
\label{r:maximal}
If the Banach space $X$ is an UMD-space, there is a deep characterization of the class $\mathcal{MR}_q(X)$ due to L. Weis; see \cite{Pruss,Weis}.\\
We only mention that, if $X$ is a UMD-space and $A\in \mathcal{H}^{\infty}(X)$, with $\mathcal{H}^{\infty}$-angle less than $\pi/2$, then $A\in \mathcal{MR}_q(X)$ for all $q\in(1,\infty)$ (see \cite{Analysis2,Pruss} for details).
\end{remark}
\subsection{Pseudodifferential Operators}
\label{ss:pseudodifferential}
Here we collect some definitions and known fact about pseudodifferential operators, see \cite{Taylor,Tools} for more on this topic.
\begin{definition} 
\label{def:pseudodiff}
Let $0\leq \delta\leq \rho\leq 1$ and $p(x,\xi)$ be a $C^{\infty}(\mathbb{R}^{2d})$ function.
\begin{itemize} 
\item[(PS)] We say  $p(x,\xi)\in S^m_{\rho,\delta}(\mathbb{R}^{2d})$, or simply $S^m_{\rho,\delta}$, if
\begin{equation}
\label{eq:pseudodiff}
|D_{\xi}^{\beta}\,D_{x}^{\alpha} p(x,\xi)| \leq C_{\alpha,\beta} \langle\xi\rangle^{m- \rho |\beta|+\delta |\alpha|}\,, \qquad \forall \alpha,\beta \in \mathbb{N}_0^d\,. 
\end{equation}
If $p$ does not depend on $x$, we say that $p(\xi)\in S^m_\rho$ if the condition in (\ref{eq:pseudodiff}) is satisfied.
\item[(E)] (Ellipticity). 
We say that $p(x,\xi)\in S^m_{\rho,\delta}$ is elliptic if
\begin{equation*}
|p(x,\xi)|\geq c  \langle\xi\rangle^m\,\qquad \text{if}\qquad|\xi|>L\,,x\in \mathbb{R}^d\,, 
\end{equation*}
for some positive constant $c$ and $L\in \mathbb{R}$ large enough.
\item[(SE)] (Strong Ellipticity). 
We say that $p(x,\xi)\in S^m_{\rho,\delta}$ is  strongly elliptic if
\begin{equation}
\label{eq:strong ellipticity p}
\Re\, p(x,\xi)\geq c  \langle\xi\rangle^m\,\qquad \text{if}\qquad |\xi|>L\,, x\in \mathbb{R}^d\,,
\end{equation}
for some positive constant $c$ and $L\in \mathbb{R}$ large enough.
\end{itemize}
Here $\langle\xi\rangle=(1+|\xi|^2)^{1/2}$.
\end{definition}
Note that, we can define a Fr\'echet topology on $S^{m}_{1,\delta}$, which is generated by the countable family of seminorm $\mathtt{S}_{j}^k$ defined as
\begin{equation}
\mathtt{S}_j^k(p):=\sup\{ \langle\xi\rangle^{-m+|\beta|-\delta|\alpha|} |D_{\xi}^{\beta}D_x^{\alpha}p(x,\xi)|\,:\, x,\xi \in \mathbb{R}^d\,,|\alpha|\leq j\,,\,\,|\beta|\leq k\}\,,
\end{equation}
for $k,j\in \mathbb{N}$. Moreover, we say that a set of symbols $\{p_i\,:\,i\in I\}$ is bounded in $S^{m}_ {1,\delta}$ if $\{\mathtt{S}_{j}^k(p_i)\,:\,i\in I\}$ is bounded in $\mathbb{R}^+$ for all $j,k\in \mathbb{N}$.\\
In the subsequent analysis, we will need of the following rough symbol class.
\begin{definition}
\label{def:pseudodiff rough}
Let $p(x,\xi)$ be a $C^{\infty}(\mathbb{R}^d;C^r_*(\mathbb{R}^d))$ map. We say $p(x,\xi)\in C^r_*S^m_{1,\delta}$ for $r>0$,  $m\in \R$, $0\leq  \delta\leq\rho\leq 1$ if
\begin{align*}
|D^{\alpha}_{\xi}p(x,\xi)|&\leq C_{\beta} \langle\xi\rangle^{m-|\alpha|}\,,\\
|D^{\alpha}_{\xi}p(\cdot,\xi)|_{C^r_*} &\leq C_{\beta} \langle\xi\rangle^{m-|\alpha|+r\delta}\,,
\end{align*}
for all $x,\xi\in \mathbb{R}^d$.
\end{definition}
Here and in the sequel, for any $f\in \mathcal{S}(\mathbb{R}^d)$, we define
\begin{equation*}
(p(x,D)f)(x)=\frac{1}{(2\pi)^d}\int_{\mathbb{R}^d} e^{i x \cdot \xi} p(x,\xi) \mathcal{F}(f)(\xi)\,d\xi\,,
\end{equation*}
where $\mathcal{F}(f)$ denotes the Fourier transform of $f$.\\
If $p(x,\xi)\in S^{m}_{1,\delta}$ with $\delta \in [0,1)$, one can prove that
\begin{equation*}
p(x,D):H^{m+s}_q \rightarrow H^{s}_q\,,
\end{equation*}
for all $s\in \mathbb{R}\,, q\in(1,\infty)$ (see \cite{Taylor}). The following fact is essential to prove the closedness of the realization of a pseudodifferential operators (see Definition \ref{def:realization}).
\begin{proposition}[A priori estimate]
\label{a priori estimate}
Let $p(x,\xi)\in S^m_{\rho,\delta}$ be an elliptic pseudodifferential operator, with $1\geq \rho>\delta\geq 0$. Then there exists a constant $c>0$ such that for all $q\in(1,\infty)$, $s\in \mathbb{R}$ and $u\in H^{s+m}_p$ we have
\begin{equation}
\label{eq:a priori estimate}
|u|_{H^{m+s}_p} \leq c(|u|_{H^s_p}+|p(x,D)u|_{H^{s}_p})\,.
\end{equation}
\end{proposition}
\begin{definition}[Realization of $p(x,D)$ on $H^s_q$] 
\label{def:realization}
Let $s>0$, $q\in (1,\infty)$ and $p(x,\xi)\in S^m_{1,\delta}$ then we will denote by $p^{R}(D)$ the closed linear operator
\begin{align*}
D(p^R(D)):=H^{s+m}_{q}\subset H^s_q &\rightarrow H^s_q\,,\\
f &\mapsto p(x,D)f\,.
\end{align*}
For notational convenience, in Definition \ref{def:realization} we have dropped the dependence on $s>0$ and $q\in(1,\infty)$.
\end{definition}
\subsection{$\mathcal{H}^{\infty}$-Calculus for Strongly Elliptic Pseudodifferential Operators}
\label{ss:Hinfinite strongly elliptic}
The goal of this section is to prove to following result, which we state as a corollary since it is an easy application of Theorem 4.1 in \cite{Hinfinite}.
\begin{corollary}
\label{cor:pseudo sectioriality strongly elliptic}
Let $p(x,\xi) \in S^2_{1,\delta}$,  with $\delta \in [0,1)$, be a strongly elliptic pseudodifferential operator and $r(x,D)\in S^{2-\varepsilon}_{1,1}$ with $\varepsilon>0$. Then for any $s\geq 0$ and $q\in(1,\infty)$, there exists $\lambda_0\in \mathbb{R}$ such that the realization of $\lambda_0 + r(x,D) + p(x,D)$ on $H^s_p$ admits a bounded $\mathcal{H}^{\infty}$-calculus on $H^s_q$, with $\mathcal{H}^{\infty}$-angle less than $\pi/2$.\\ 
Furthermore, if $c,L$ are as in (SE) of Definition \ref{def:pseudodiff}, we can choose the value $\lambda_0$ depending only on $L,c$ and $\mathtt{S}_j^k(p)$ for $j,k$ sufficiently large.
\end{corollary}
\begin{proof}
Combining Theorem 4.1 and Example 4.2 in \cite{Hinfinite} (recall $F^s_{q,2}=H^s_q$), we have to prove that:
\begin{itemize}
\item $r(x,D):H^{2-\nu+s}_q \rightarrow H^s_q$, for $\nu>0$ small enough.
\item $p(x,\xi)\in S^2_{1,\delta}$ is $\Lambda_{\phi}$-elliptic, with 
\begin{equation*}
\Lambda_{\phi}:=\{re^{i\theta}\,:\,r\geq 0\,,\, \phi<\theta<2\pi-\phi\}\,,
\end{equation*}
for some $\phi<\pi/2$; for the definition of $\Lambda_{\phi}$-ellipticity we refer to \cite{Hinfinite} Definition 3.3.
\end{itemize}
The first item is quite easy, since by Theorem 9.1 in Chapter 13 in \cite{Taylor} it follows that
\begin{equation*}
r(x,D):H^{2-\varepsilon+s}_q \rightarrow H^s_q\,, 
\end{equation*}
provided $s>0$ and $q\in(1,\infty)$.  Now to prove the first item choose $\nu=\varepsilon$ if $s>0$, for $s=0$ one can choose $\nu=\varepsilon/2$, since
\begin{equation*}
r(x,D):H^{2-\varepsilon+\varepsilon/2}_q \rightarrow H^{\varepsilon/2}_q\hookrightarrow L_q\,. 
\end{equation*}
The second item is a verification of the hypotheses (H1)-(H2) of Definition 3.3 in \cite{Hinfinite}.\\
Before proceeding further, we note that for each $x\in \mathbb{R}^d$ and $|\xi|>L$ we have $p(x,\xi)\in \Sigma_{\phi'}$ for a $\phi'\in (\arctan(C_{0,0}/c);\pi/2)$; here $\Sigma_{\phi'}$ is the sector symmetric with respect to the positive real axis of angle $2\phi'$.\\
Indeed, since $p(x,\xi)$ is strongly elliptic (see (\ref{eq:strong ellipticity p})), then for each $x\in \mathbb{R}^d$ and $|\xi|>L$ we have
\begin{equation}
\label{eq:range}
|\Im p(x,\xi)| \leq C_{0,0}\langle\xi\rangle^{2} \leq \frac{C_{0,0}}{c} \Re p(x,\xi)\,,
\end{equation}
and the claim follows; here $C_{0,0}$ is as in (\ref{eq:pseudodiff}) for $\alpha=\beta=0$.\\
In order to verify the hypothesis (H1), we have to find a $R>0$ such that for all $x\in \mathbb{R}^d$ and $|\xi|>R$, we have
\begin{equation*}
\sigma(p(x,\xi))\subset \{z\in \mathbb{C}\,:\, C_0^{-1}\langle\xi\rangle^{2}<|z|< C_0\langle\xi\rangle^{2}\,,\,\,z\notin \Lambda_{\phi}\}\,,
\end{equation*}
for some $C_0>0$ and $\phi\in(0;\pi/2)$; here $\sigma(\cdot)$ denotes the spectrum.\\
Since $p(x,\xi)$ is scalar and (\ref{eq:range}) holds, we have only to verify that $C_0^{-1}\langle\xi\rangle^{2}<|p(x,\xi)|< C_0\langle\xi\rangle^{2}$ for some $C_0$. By the strongly ellipticity assumption and the fact $p(x,\xi)\in S^{2}_{1,\delta}$, this holds for $R=L$ and $C_0=\max\{C_{0,0},c^{-1}\}$.\\
We now move to hypothesis (H2). Note that, with simple geometrical considerations, if $\lambda\in \Lambda_{\phi}$ where $\pi/2>\phi>\phi'$ and $\pi>\phi+ \phi'$, we obtain 
\begin{equation*}
|-\lambda + p(x,\xi)|\geq c' (|\lambda|+|p(x,\xi)|)\geq c' (\Re p(x,\xi))\geq c' c \langle\xi\rangle^{2}\,,
\end{equation*}
 where $c'=\cos((\pi-\phi-\phi')/2)$ and $|\xi|>L$. Now, for any $\alpha,\beta\in \mathbb{N}^d_0$, it is easy to see that
 \begin{align*}
 |D^{\alpha}_{\xi} D^{\beta}_{x} p(x,\xi)| |\lambda-p(x,\xi) |^{-1} &\leq (C_{\alpha,\beta} \langle\xi\rangle^{2-|\alpha|+\delta |\beta|})(c c' \langle\xi\rangle^{2})^{-1}\\
 &= \tilde{C} \langle\xi\rangle^{-|\alpha|+\delta |\beta|}\,,
 \end{align*}
 as desired.
\end{proof}
In the future analysis we do not need the $s=0$ case of Corollary \ref{cor:pseudo sectioriality strongly elliptic}.\\
Lastly, we mention that the same application of Theorem 4.1 in \cite{Hinfinite} shows that the same result holds with $H^s_q$ replaced by $F^s_{q,q'}$ or $B^{s}_{q,q'}$ where $q,q'\in (1,\infty)$ and $s\geq 0$. But we do not need this in the following.\section{Paradifferential Operator}
\label{s:paradifferential}
\subsection{Introduction and Proof of Property (P2)}
\label{ss:Introduction and P2}
Before starting with the proofs, we recall the following fairly simple construction.\\
Take any $\Psi_0\in C^\infty_0(\mathbb{R}^d)$ such that $\Psi(\xi)=1$ on $|\xi|\leq 1/2$ and $\Psi(\xi)=0$ on $|\xi|\geq 1$, then define
\begin{align}
\label{eq:Psi}
\Psi_k(\xi)&:=\Psi_0(2^{-k}\xi)\,,&\qquad&k\geq 1\,,\\
\label{eq:psi}
 \qquad \psi_k(\xi)&:=\Psi_k(\xi)-\Psi_{k-1}(\xi)\,, &\qquad&k\geq 1\,,\\
 \psi_0(\xi)&:=\Psi_0(\xi)\,. &\qquad&
\end{align}
It is easy to see that
\begin{equation}
\label{eq:partition}
\sum_{k=0}^{\infty} \psi_k(\xi)=1\,.
\end{equation}
In light of (\ref{eq:partition}) this construction is called (smooth) Littlewood-Paley partition of the unity.\\
\\
We start with a simple observation. If $|\alpha|\leq 2$ and $u\in C^2:=C^2(\R^d)$, then
\begin{equation}
\label{eq:pointwise convergence}
\lim_{k\nearrow \infty} D^{\alpha}\Psi_k(D)u(x) = D^{\alpha}u(x)\,, 
\end{equation}
for all $x\in \R^d$. Indeed, $D^{\alpha}\Psi_k(D)u=\F^{-1}(\Psi_k)\star D^{\alpha}u$, then $\F^{-1}(\Psi_k)$ is an approximation of the identity, since $D^{\alpha}u\in C(\R^d)$ every point in $\R^d$ is a Lebesgue point so the claim follows by the standard result, see for instance \cite{Analysis1}.\\ 
For reader's convenience, we sketch the construction of paradifferential operator  in order to highltight the quasilinear structure arising from this construction (see Chapter 13 Section 10 of \cite{Taylor} for a more complete treatment).\\ 
Consider $F$ as in (H3), (for notational convenience, we do not report explicity the dependence of $F$ on $x$, so we write $F(\zeta)$ instead of $F(x,\zeta)$) and let $u\in C^{2}_*$. Then by (\ref{eq:pointwise convergence}),
\begin{multline}
\label{eq:paradiff 1}
F(D^2 u)= F(D^2\Psi_0(D)u) + [F(D^2\Psi_1(D)u) -F(D^2\Psi_0(D)u)] + \dots\\
[F(D^2\Psi_{k+1}(D)u) -F(D^2\Psi_k(D)u)]+\dots\,,
\end{multline}
pointwise. Now, it easy to see that
\begin{multline*}
[F(D^2\Psi_{k+1}(D)u) -F(D^2\Psi_k(D)u)] =\\
 \sum_{|\alpha|\leq 2} \int_0^1 dt\, \frac{\partial F}{\partial \zeta_{\alpha}}(\Psi_k(D)D^2u + t\psi_{k+1}(D)D^2 u)\, (\psi_{k+1 }(D) D^\alpha u)\,.
\end{multline*}
Define
\begin{align*}
m_k^{u,\alpha}(x)&= \int_0^1 dt\, \frac{\partial F}{\partial \zeta_{\alpha}}(\Psi_k(D)D^2u + t\psi_{k+1}(D)D^2 u)\,,\\
M_F^{\alpha}(u;x,\xi)&=\sum_{k=0}^{\infty} m_k^{u,\alpha}(x)\psi_{k+1}(\xi)\xi^{\alpha}\,\\
M_F(u;x,\xi)&= \sum_{|\alpha|\leq 2} M_F^{\alpha}(u;x,\xi)\,. 
\end{align*}
Due to (\ref{eq:paradiff 1}) then
\begin{equation}
\label{eq:pointwise equality}
F(D^2u)=F(D^2\Psi_0(D)u) + M_F(u;x,D)u\,,
\end{equation}
pointwise. Before proving property (P2) we note that, by the assumption (H1) we have $q>(2+d)/s$, then
\begin{equation}
\label{eq:regularity for trace space}
(H^s_q,H^{s+2}_q)_{1-\frac{1}{q},q}= W^{s+2-\frac{2}{q}}_{q}\hookrightarrow C_{*}^{2+r}	\,,
\end{equation}
where $r:=s-(2+d)/q>0$; the last injection follows by Sobolev embedding Theorem.\\
Now we are in the position to prove the condition (P3); we recall that $G(u):=-F(D^2\Psi_0(D)u)$.
\begin{proof}[Proof of property (P2)] To prove that $F$ maps $W^{s+2-\frac{2}{q}}_{q}$ into $H^s_q$, since by hypothesis (H3) $G(0)=-F(x,0)\in H^s_q$, it is enough to prove the Lipschitz continuity.\\
In order to estimate $|G(u)-G(v)|_{H^s_q}$ one can use the Proposition 7.1 of Chapter 2 in \cite{Tools}, i.e.
\begin{align*}
|F(D^2\Psi_0(D)&u)- F(D^2\Psi_0(D)v)|_{H^s_q}\\
 &\leq K(|D^2\Psi_0(D)u|_{L_{\infty}},|D^2\Psi_0(D)v|_{L_{\infty}})\\
 &\cdot(1+|D^2\Psi_0(D)u|_{H^s_q}+|D^2\Psi_0(D)v|_{H^s_q}) |D^2\Psi_0(u-v)|_{L_{\infty}}\\
 &+C|G(D^2\Psi_0(D)u,D^2\Psi_0(D)v)|_{L_{\infty}} |D^2\Psi_0(D)(u-v)|_{H^s_q}\,;
\end{align*}
where $K:\mathbb{R}^+\times \mathbb{R}^+\rightarrow \mathbb{R}^+$ is locally bounded and
\begin{equation*}
G(u,v)= \sum_{|\alpha|\leq 2} \int_{0}^1 dt \frac{\partial F}{\partial \zeta_{\alpha}} \Big(D^2\Psi_0(D)((1-t)u+tv))\Big)\,.
\end{equation*}
Fix any $n\in \mathbb{N}$. Since $D^{\alpha}\Psi_0(D)\in S^{-N}$ for any $N\in \N$ and $|\alpha|\leq 2$, then for suitable $\tilde{C}_n,C_n>0$,
\begin{align*}
|D^2\Psi_0(D)u|_{L_{\infty}}&\leq C_n |u|_{L_{\infty}} <\tilde{C}_n\,,\\
|D^2\Psi_0(D)v|_{L_{\infty}}&\leq C_n |v|_{L_{\infty}} <\tilde{C}_n\,,\\
|D^2\Psi_0(D)u|_{H^s_q} &\leq C_n |u|_{W^{s+2-\frac{2}{q}}_{q}}<\tilde{C}_n\,,\\
|D^2\Psi_0(D)v|_{H^s_q} &\leq C_n |v|_{W^{s+2-\frac{2}{q}}_{q}}<\tilde{C}_n\,,\\
|D^2\Psi_0(u-v)|_{L_{\infty}}&\leq C_n |u-v|_{L_{\infty}}  \leq C_n |u-v|_{W^{s+2-\frac{2}{q}}_{q}}\,,\\
|D^2\Psi_0(D)(u-v)|_{H^s_q}&\leq C_n|u-v|_{W^{s+2-\frac{2}{q}}_{q}}\,,
\end{align*}
for any $u,v\in B_{s,q}(0,n)$.
\end{proof}
\subsection{Proof of Proposition \ref{prop:Hinfinite M} and Property (P1)}
\label{ss:proof of Propositon and P1}
In this subsection we analyse the pseudodifferential operator $M_F(u;x,\xi)$; recall that $u\in W^{s+2-2/q}_{q}\hookrightarrow C^{2+r}_*$ by (\ref{eq:regularity for trace space}).\\
Now we recall the following, proven in \cite{Taylor}, Chapter 13 Section 10:
\begin{equation*}
M_F(u;x,\xi)\in S^{2}_{1,1}\cap C^r_* S^2_{1,0}\,.
\end{equation*}
At this point, we take advantage of the smoothing symbol technique proposed in \cite{Taylor,Tools}. For any $\delta\in (0,1)$ this technique allows us to write 
\begin{equation}
\label{eq:MF decomposition}
M_F(u;x,\xi)=M^{\sharp}_{u}(x,\xi)+M^b_{u}(x,\xi)\,,
\end{equation}
where
\begin{equation}
\label{eq:M decomposition regularity}
M^{\sharp}_{u}\in S^2_{1,\delta}\,, \qquad M^b_{u}\in S^{2-r\delta}_{1,1}\cap C^r_*S^{2-r\delta}_{1,\delta}\,.
\end{equation}  
Furthermore $M^{\sharp}_{u}$ is explicitly given by
\begin{equation}
\label{def:M sharp}
M^{\sharp}_{u}(x,\xi)=\sum_{k=0}^{\infty} J_{\varepsilon_k}(M_F(u;x,\xi))\,\psi_{k}(\xi)\,,
\end{equation}
where $\varepsilon_k=2^{-k\delta}$, $\delta\in (0,1)$, $J_{\varepsilon}=\tau(\varepsilon D)$ acts on the variable $x$ and $\tau \in C^{\infty}_0(\mathbb{R}^d)$ such that $\tau(\xi)=1$ for $|\xi|<1$; see \cite{Taylor,Tools}.\\
Exploiting the construction (\ref{eq:MF decomposition})-(\ref{def:M sharp}) we will prove two lemmas, which permits to demonstrate Proposition \ref{prop:Hinfinite M}.
\begin{lemma}
\label{l:uniform boundness}
For any $n\in \mathbb{N}$ and $\delta\in (0,1)$, the sets $\{M_{u}^{\sharp}\,:\,u\in B_{s,q}(0,n)\}$ and $\{M_{u}^b\,:\,u\in B_{s,q}(0,n)\}$ are bounded respectively in $S^{2}_{1,\delta}$ and in $S^{2-r\delta}_{1,1}$.
\end{lemma}
\begin{proof}
Note that if $u\in B_{s,q}(0,n)\subset C^{2+r}_{*}$ (here, as before, $r:=s-(2+d)/q>0$ by (H2)) then by Sobolev embedding Theorem
\begin{equation}
\label{eq:uniform boundness in W}
|u|_{C^{2+r}}\leq C |u|_{W^{s+2-2/q}_q}\leq C_n\,,
\end{equation} 
for a suitable $C_n$. By this and the analysis of paradifferential operator done in Section 10 Chapter 13 in \cite{Taylor}, it is easy to see that $$\{M_u(x,\xi)\,:\,u\in B_{s,q}(0,n)\}\,\, \text{is bounded in}\,\,C^{r}_*S^2_{1,0}\cap S^{2}_{1,1}\,.$$
Due to Proposition 10.4-10.5 in \cite{Taylor} Chapter 13, the claim follows.
\end{proof}
\begin{lemma}
\label{l:uniform strongly}
For any $n\in \mathbb{N}$, there exists $\lambda_n\in \mathbb{R}^+$ such that $\lambda_n+M_{F}^{\sharp}(u;x,\xi)\in S^{2}_{1,\delta}$ is strongly elliptic for any $u\in B_{s,q}(0,n)$.\\
Furthermore, there exist $c_n,L_n>0$, such that
\begin{equation}
\label{eq:uniformly strongly}
M_{F}^{\sharp}(u;x,\xi)\geq c_n |\xi|^{2}\,,
\end{equation}
for any $x\in \mathbb{R}^d$, $|\xi|>L_n$ and $u\in B_{s,q}(0,n)$.
\end{lemma}
\begin{proof}
As in the proof of Lemma \ref{l:uniform boundness}, we have $|u|_{C^{2+r}}\leq C_n$. For clarity we divide the proof in two steps.\\
\textit{Step 1 - $M_F(u;x,\xi)$ is strongly elliptic and satisfy an estimate similar to (\ref{eq:uniformly strongly})}.
Since $|\Psi_k(D)D^2u+\psi_k(D)u|_{L_{\infty}}\leq \tilde{C}_n$ for any $t\in[0,1]$ and $u\in B_{s,q}(0,n)$, we have
\begin{align*}
\tilde{M}_F&(u;x,\xi) :=\sum_{|\alpha|=2} \sum_{k=0}^{\infty} m_{k}^{u,\alpha}(x)\xi^{\alpha}\psi_{k+1}(\xi)\\
&=\sum_{k=0}^{\infty}\int_0^1 dt\,\left( \sum_{|\alpha|=2}  \frac{\partial F}{\partial \zeta_{\alpha}}(\Psi_k(D)D^2u + t\psi_{k+1}(D)D^2 u)\xi^{\alpha}\right)  \psi_{k+1}(\xi)\\
&\geq c'_n \langle\xi\rangle^{2}\,,
\end{align*}
by hypothesis (H4). Furthermore
\begin{align*}
|&M_{F}(u;x,\xi)- \tilde{M}_{F}(u;x,\xi)|\\
&= \sum_{|\alpha|\leq 1}\sum_{k=0}^{\infty} \left(\int_0^1 dt\,   \frac{\partial F}{\partial \zeta_{\alpha}}(\Psi_k(D)D^2u + t\psi_{k+1}(D)D^2 u)\right)\xi^{\alpha}  \psi_{k+1}(\xi)\\
&\leq \sup_{|\zeta|\leq \tilde{C}_n}|D^1 F(\zeta)|\,\langle\xi\rangle\,. 
\end{align*}
Set $M_n:=\sup_{|\zeta|\leq \tilde{C}_n}|D^1 F(\zeta)|$, we have
\begin{align*}
M_{F}(u;x,\xi)&= \tilde{M}_{F}(u;x,\xi) + (M_{F}(u;x,\xi)-\tilde{M}_{F}(u;x,\xi))\\
&\geq c'_n  \langle\xi\rangle^{2} - M_n \langle\xi\rangle= \langle\xi\rangle^{2}\left( c'_n - \frac{M_n}{\langle\xi\rangle}  \right)\geq \frac{c'_n}{2}\langle\xi\rangle^{2}\,,
\end{align*}
for $\langle\xi\rangle\geq 2M_n/c'_n$.\\
\textit{Step 2 - Conclusion}. Note that
\begin{align*}
|M^{\sharp}_u(x,\xi) - M_F(u;x,\xi)| &=\sum_{j=0}^{\infty}  \Big(J_{\varepsilon_{j}}(M_F(u;x,\xi))-M_F(u;x,\xi)\Big)\\
&= 
\sum_{j,k=0}^{\infty} \sum_{|\alpha|\leq 2}  \Big(J_{\varepsilon_j}m_{k}^{u,\alpha}-m_{k}^{u,\alpha}\Big)\psi_{k+1}(\xi)\psi_j(\xi)\xi^{\alpha}\,.
\end{align*}
Furthermore the set $\{m_{k}^{u,\alpha}\,:\,u\in B_{s,q}(0,n)\,:\,|\alpha|\leq 2\}$ is bounded in $C^r_{*}$, since $|\Psi_k(D)D^2u+\psi_k(D)u|_{L_{\infty}}\leq \tilde{C}_n$ for any $t\in[0,1]$ and $u\in B_{s,q}(0,n)$. \\
By Lemma 9.8 in Chapter 13 of \cite{Taylor}, we have $|J_{\varepsilon_j}m_{k}^{u,\alpha}-m_{k}^{u,\alpha}|_{L_{\infty}}\leq C_n \varepsilon_j^r$; where $C_n\in \mathbb{R}^+$ depends only on $n\in \mathbb{N}$.\\
Recall that $\varepsilon_j=2^{-j}$ and $\text{supp}\,\psi_j \sim 2^j$, this implies
\begin{align*}
|M^{\sharp}_u(x,\xi) - M_F(u;x,\xi)| &= \sum_{j,k=0}^{\infty} \sum_{|\alpha|\leq 2}  \Big(J_{\varepsilon_j}m_{k}^{u,\alpha}-m_{k}^{u,\alpha}\Big)\psi_{k+1}(\xi)\psi_j(\xi)\xi^{\alpha}\\
&\leq \sum_{j=0}^{\infty}  C_n 2^{-j r} \langle\xi\rangle^2\psi_j(\xi) \leq \tilde{R}_n \langle \xi\rangle^{2-r}\,;
\end{align*}
for $\tilde{R}_n\in \mathbb{R}^+$ suitable. Writing 
$$M^{\sharp}_u(x,\xi) = M_F(u;x,\xi) + (M^{\sharp}_u(x,\xi) - M_F(u;x,\xi))\,,$$ 
with the same argument performed at the end of \textit{Step 1}, one can easily conclude the proof.
\end{proof}
With this in hands, we can turn to the proof of Proposition \ref{prop:Hinfinite M}.
\begin{proof}[Proof of Proposition \ref{prop:Hinfinite M}]
Due to Lemmas \ref{l:uniform boundness}-\ref{l:uniform strongly} and the decomposition 
$$M_{F}(u;x,\xi)=M_{u}^{\sharp}(x,\xi)+M_u^b(x,\xi)\,,$$ 
the claim follows by Corollary \ref{cor:pseudo sectioriality strongly elliptic}.
\end{proof}
We turn to the proof of property (P1). To do this, we recall the following result, extracted by a more general result proven in \cite{Taylor}; for the sake of completeness we sketch the proof.
\begin{proposition}
\label{prop:bound 1,1}
Let $p(x,\xi)\in S^m_{1,1}$ be an elementary symbol, i.e.
\begin{equation}
\label{eq:elementary}
p(x,\xi)=\sum_{k=0}^{\infty} Q_k(x)\varphi_k(\xi)\,,
\end{equation}
where $\varphi_{k}(\xi)$ is a bounded sequence in $S^m_{1}$ and $\text{supp}\,\varphi_k\subset \{\xi\,:\,a2^{k-1}<|\xi|<a2^{k+1}\}$ for some $a>0$. Then for any $s>0$ and $q\in (1,\infty)$
\begin{equation*}
p(x,D): H^{m+s}_q\rightarrow H^s_q\,,
\end{equation*}
with operator norm bounded by
\begin{equation}
\label{eq:norm elementary bound}
C_{s,m,q} \Big\{\sup_{k\geq0} |Q_k|_{L_{\infty}} + \sup_{k\geq0}  2^{- l k}|Q_k|_{C^l_*} \Big\}\,,
\end{equation}
for all $q\in(1,\infty)$ and $0<s < l$; here $C_{s,m,q}$ denotes a positive constant which depends only on $s,m,q$ and the sequence $\varphi_k$ in (\ref{eq:elementary}).
\end{proposition} 
\begin{proof}[Sketch of the proof]
Consider the pseudodifferential operator $\Lambda^{\mu}$ with symbol $\langle\xi\rangle^{\mu}\in S^{\mu}_1$, for any $\mu \in \mathbb{R}$. Furthermore, it is easy to see that $\Lambda^{\mu}$ is an isometric isomorphism between $H^{s+\mu}_q$ and $H^{s}_q$, for all $s\in \mathbb{R}$ and $q\in(1,+\infty)$. By this, it is enough to prove the $m=0$ case of Proposition \ref{prop:bound 1,1}. Indeed, suppose $m\neq 0$ and $p(x,\xi)\in S^m_{1,1}$ is an elementary symbol of the form (\ref{eq:elementary}), then the symbol
\begin{equation*}
\tilde{p}(x,\xi)=p(x,\xi)\langle\xi\rangle^{-m} =\sum_{k=0}^{\infty} Q_k(x)\varphi_k(\xi)\langle\xi\rangle^{-m}\,,
\end{equation*}
belongs to $S^0_{1,1}$ and it is elementary. If the claim in the Proposition \ref{prop:bound 1,1} is valid for $m=0$ then 
\begin{multline}
\label{eq:proof 1,1}
|p(x,D)\Lambda^{-m}u|_{H^{s}_q}=|\tilde{p}(x,D)u|_{H^{s}_q} \\
\leq C_{s,m,q} \Big\{\sup_{k\geq0} |Q_k|_{L_{\infty}} + \sup_{k\geq0}  2^{- l k}|Q_k|_{C^l_*} \Big\}|u|_{H^s_q}\,.
\end{multline}
Now, for any $u\in H^{s}_q$ then $\Lambda^{-m}u=v\in H^{m+s}_q$ and $|v|_{H^{s+m}_q}=|u|_{H^s_q}$; using this in (\ref{eq:proof 1,1}) we obtain the claim.\\
The $m=0$ case follows by the analysis in \cite{Taylor} vol III pp. 52-54.
\end{proof}
Now we can prove the property (P1).
\begin{proof}[Proof of property (P1)]
It is sufficient to prove that, for any $u,v\in V$ and $\alpha$ such that $|\alpha|\leq 2$, the pseudodifferential operator with symbol
\begin{equation}
\label{eq:M difference}
M_F^{\alpha}(u;x,\xi) - M_{F}^{\alpha}(v;x,\xi) = \sum_{k=0}^{\infty} (m_{k}^{u,\alpha}-m_k^{v,\alpha}) \psi_{k+1}(\xi)\xi^{\alpha}\,,
\end{equation}
maps $H^{s+2}_q$ in $H^{s}_q$ with operator norm bounded by $C_{s,q}|u-v|_{C^{2+r}_*}$, since $W^{s+2-2/q}_q\hookrightarrow C^{2+r}_*$.\\
It is easy to see that the symbol in (\ref{eq:M difference}) is an elementary symbol, in virtue of Proposition \ref{prop:bound 1,1}, we have to prove the existence of a constant $C$ such that
\begin{align}
\label{bound mk 1}
|m_{k}^{u,\alpha}-m_k^{v,\alpha}|_{L_{\infty}}\leq C|u-v|_{C^2}\,,\\
\label{bound mk2}
2^{- l k}|m_{k}^{u,\alpha}-m_k^{v,\alpha}|_{C^l}\leq C|u-v|_{C^2}\,,
\end{align}
for all $k,l\in \mathbb{N}$; since $C^2\hookrightarrow C^{2}_*$.\\ 
Now, we rewrite $m_{k}^{u,\alpha}-m_k^{v,\alpha}$ in a convenient way:
\begin{multline}
\label{eq:M coefficient difference}
m_{k}^{u,\alpha}-m_k^{v,\alpha} = \\ \int_{0}^1 dt \int_{0}^1 ds \sum_{|\beta|\leq 2} \frac{\partial F}{\partial \zeta_{\alpha}\partial \zeta_{\beta}}(\lambda^{u,v}_k(s,t))
\cdot(D^{\beta}\Psi_k(D)(u-v)+t\psi_{k+1}(D)D^{\beta}(u-v))\,;
\end{multline}
where, for brevity, we have set
\begin{multline*}
\lambda^{u,v}_k(s,t):=(D^2\Psi_k(D)v+t\psi_{k+1}(D)D^2v)\\
 + s(D^2\Psi_k(D)(u-v)+t\psi_{k+1}(D)D^2(u-v))\,,
\end{multline*}
and $(\lambda^{u,v}_k(s,t))_{\beta}=:\prescript{\beta}{}{\lambda^{u,v}_k(s,t)}$ for all $|\beta|\leq 2$.
By (\ref{eq:Psi})-(\ref{eq:psi}) and Young inequality, we have 
\begin{align}
\label{eq:bound 1}
|\lambda^{u,v}_k(s,t)|_{L_{\infty}}
&\leq C(|u|_{C^2}+|v|_{C^2})\,,\\
\label{eq:bound 2}
|D^{\beta}\Psi_k(D)(u-v)+t\psi_{k+1}(D)D^{\beta}(u-v)|_{L_{\infty}}&\leq C(|u-v|_{C^2})\,,
\end{align}
for all $|\beta|\leq 2$ and $s,t\in [0,1]$. Since $V\subset C^{2+r}_*$ is bounded, we have the values $|u|_{C^l},|v|_{C^l}$ are uniformly bounded, so (\ref{bound mk 1}) follows easily by the smoothness hypothesis on $F$ in Theorem \ref{t:well posedness}.\\
Now we move to the proof of (\ref{bound mk2}), for convenience we prove (\ref{bound mk2}) for $l=1$; the general cases follow in the same manner.\\
Take any $j \in \{1,\dots,d\}$, by (\ref{eq:M coefficient difference}) and Leibniz rule
\begin{multline}
\label{eq:M difference derivative}
D_{x_j} (m_{k}^{u,\alpha}-m_k^{v,\alpha}) = \\ \int_{0}^1 dt \int_{0}^1 ds \sum_{|\beta|\leq 2} D_{x_j}\left(\frac{\partial F}{\partial \zeta_{\alpha}\partial \zeta_{\beta}}(\lambda^{u,v}_k(s,t))\right)
\cdot(\varrho^{u,v}_k(t)) 
\\+ \int_{0}^1 dt \int_{0}^1 ds \sum_{|\beta|\leq 2} \frac{\partial F}{\partial \zeta_{\alpha}\partial \zeta_{\beta}}(\lambda^{u,v}_k(s,t))
\cdot D_{x_j}(\varrho^{u,v}_k(t))	\,;
\end{multline}
where
\begin{equation*}
\varrho^{u,v}_k(t):=D^{\beta}\Psi_k(D)(u-v)+t\psi_{k+1}(D)D^{\beta}(u-v)\,.
\end{equation*}
For the second term in the RHS of (\ref{eq:M difference derivative}) we can use the bound in (\ref{eq:bound 1}) and 
\begin{equation}
\label{eq:u higher bound}
|D_{x_j} [D^{\beta}\Psi_k(D)(u-v)+t\psi_{k+1}(D)D^{\beta}(u-v)]|_{L_{\infty}}\leq C2^{k}(|u-v|_{C^2})\,,
\end{equation}
another time by Young inequality.\\
For the first terms in the RHS of (\ref{eq:M difference derivative}) one can use the composition rules, and obtains
\begin{equation*}
D_{x_j}\left(\frac{\partial F}{\partial \zeta_{\alpha}\partial \zeta_{\beta}}(\lambda^{u,v}_k(s,t))\right) =
\sum_{|\mu|\leq 2} \frac{\partial F}{\partial \zeta_{\alpha}\partial \zeta_{\beta} \partial\zeta_{\mu}}(\lambda^{u,v}_k(s,t))\cdot D_{x_j}[\prescript{\mu}{}{\lambda^{u,v}_k(s,t)}]\,.
\end{equation*} 
To bound the previous term, note 
\begin{equation*}
|D_{x_j}(D^{\mu}\Psi_k(D)u+t\psi_{k+1}(D)D^{\mu}u)|_{L_{\infty}} \leq C 2^{k}|D^2 u|_{L_{\infty}}\,.
\end{equation*}
Using the previous bound, the inequality in (\ref{eq:bound 1}), (\ref{eq:u higher bound}) and the smoothness hypothesis on $F$, we obtain (\ref{bound mk2}).
\end{proof}
\section{Proof of the Main Results}
\label{s:proof of the main results}
\subsection{Proof of Theorem \ref{t:well posedness} and Parabolic Regularization}
\label{ss:deterministic proof}
The Proof of Theorem \ref{t:well posedness} is based on the abstract framework developed in \cite{Pruss} Chapter 5; in this section we will largely follow its exposition. In \textit{Step 2} of the following, we use the same argument of the proof of Corollary 5.1.2 in \cite{Pruss}.
\begin{proof}[Proof of Theorem \ref{t:well posedness}]
For reader's convenience, we divide the proof into two steps.\\
\textit{Step 1 - Local Existence}. As explained in Section \ref{s:introduction} the parabolic problem (\ref{EQ}) is reduced to the following abstract quasilinear parabolic PDEs 
\begin{equation}
\label{eq:proof recall equation}
u'+M_F(u)u =G(u)\,,\,\,t>0\,,\qquad u(0)=u_0\,.
\end{equation}
Fix any $n$ such that $u_0\in B_{s,q}(0,n)$, thanks to Proposition \ref{prop:Hinfinite M} and Remark \ref{r:maximal}, there exists a $\lambda_n\in \mathbb{R}$ such that $\lambda_n + M_{F}(u)\in\mathcal{MR}_q(H^s_q)$ (recall that $H^s_q$ are UMD-space for all $s\in \mathbb{R}$ and $q\in(1,\infty)$). Due to the properties (P1)-(P2) we can apply Theorem 5.1.1 in \cite{Pruss} and conclude the existence of a unique local solution to (\ref{eq:proof recall equation}) in
\begin{equation}
\label{eq:regularity class}
H^1_{q}(0,T;H^s_q)\cap L_{q}(0,T;H^{2+s}_q)\cap C([0,T];W^{2+s-\frac{2}{q}}_{q})\,;
\end{equation}
for a suitable $T>0$.\\ 
Furthermore, for each $u_0\in W^{s+2-2/q}_q$, there exists $t(u_0)>0$ and $\varepsilon(u_0)>0$ such that for any $v\in B_{s,q}(u_0,\varepsilon)$ the solution to (\ref{eq:proof recall equation}) with initial data $v$ exists on $[0,t(u_0)]$. \\
\textit{Step 2 - Maximal defined solution and blow up criterion}. Define 
\begin{equation*}
T(u_0):=\sup\{a>0\,:\, \text{(\ref{eq:proof recall equation}) has a solution on}\,\, [0,a]\,\,\text{in (\ref{eq:regularity class})} \}\,.
\end{equation*}
If $T(u_0)=\infty$ we have nothing to prove, for this suppose $T(u_0)<\infty$ and (\ref{eq:lim extended}) does not hold. So $\lim_{t\nearrow T(u_0)} u(t)$ exists in $W^{s+2-2/q}_q$, in particular the set $u([0,T(u_0)])$ is a compact subset of $W^{s+2-2/q}_q$. By the previous step and an easy compactness argument, it easy to see that there exists a $\delta>0$ such that the following problem
\begin{equation}
\label{eq:problem blow up}
v'+M_F(v)=G(v)\,,\,\,t>0\,,\qquad v(0)=u(s)\,;
\end{equation}
has a solution in $H^1_{q}(0,\delta;H^s_q)\cap L_{q}(0,\delta;H^{2+s}_q)\cap C([0,\delta];W^{2+s-\frac{2}{q}}_{q})$ for any $s\in[0,T(u_0)]$. Take $t'$ such that $T(u_0)-\delta<t'<T(u_0)$, then solution $v(t)$ to the problem (\ref{eq:problem blow up}) coincides with $u(t+t')$ and extends the solution beyond $T(u_0)$ and this contradicts the definition of $T(u_0)$.
\end{proof}
By hypothesis (H3) one can guess that the solution provided by Theorem \ref{t:well posedness} is more regular than (\ref{eq:regularity class}) in a sense clarified below, in other words, one has the parabolic regularization of the solution.\\
In this direction we state the following proposition (we omit the proof in this paper):
\begin{proposition}
\label{prop:regularity M,G}
Under the hypotheses (H2)-(H4), the maps
\begin{align*}
W^{s+2-2/q}_q\ni u &\mapsto M_{F}(u)\in \B(H^{s+2}_q,H^s_q)\,,\\
W^{s+2-2/q}_q\ni u &\mapsto G(u)\in H^s_q\,,
\end{align*}
are of class $C^{\infty}$ between the indicated spaces.
\end{proposition}
\begin{theorem}[Parabolic Regularization]
Suppose that the hypotheses (H2)-(H4) are satisfied. Then there exist $r>0$ and $T>0$ such that the map
\begin{equation*}
\psi : B_{s,q}(u_0,r) \rightarrow C^{\infty}((0,T);H^{s+2}_q)\,, \qquad \psi(v)(\cdot)=u(\cdot,v)\,;
\end{equation*}
is of class $C^{\infty}$; where $u(\cdot,v)$ is the solution to (\ref{EQ}) with initial data $v$ provided by Theorem \ref{t:well posedness}. In particular, $u(\cdot,u_0)\in C^{\infty}((0,T(u_0));H^{s+2}_q)$, for all $u_0 \in W^{s+2-2/q}_q$.
\end{theorem}
\begin{proof}
We begin the proof recalling that, as showed in the proof of Theorem \ref{t:well posedness}, for each $u_0\in W^{s+2-2/q}_q$ there exists an $r>0$ such that the \textit{local} solution of (\ref{EQ}) for initial data $v\in B_{s,q}(u_0,r)$ exists on an interval $[0,T]$ independent on $v$; so the map $\psi$ is well defined for $T>0$ small enough.\\
The claim now follows by Proposition \ref{prop:regularity M,G} and Theorem 5.2.1 in \cite{Pruss}.
\end{proof}
In the last part of this section, we want to weaken the hypothesis (H4), replacing it with the following:
\begin{itemize}
\item[(H4')] There exists an open subset $V$ of $\R\times i\R^d\times \s$, such that for each $n\in \mathbb{N}$ there exists $c'_n>0$ such that
\begin{equation*}
\sum_{|\alpha|=2}\frac{\partial F}{\partial \zeta_{\alpha}}(x,\zeta)\, \xi^{\alpha} \geq c'_n |\xi|^2\,,\qquad \forall \xi
\in \mathbb{R}^d\,,
\end{equation*}
uniformly in $x\in \mathbb{R}^d$ and $\zeta\in  \{|\zeta|\leq n\}\cap  V$.
\end{itemize}
In order to use the Maximal $L_q$-regularity results, we will need a strengthening of hypothesis (H1). To do this, we denote with $R(u):=\{D^2u(x)\,:\,x\in \R^d\}$ the range of $u\in W^{s+2-2/q}_q$ and $\d(U,W):=\inf\{|x-y|\,:\,x\in U\,,\, y\in W\}$, for $U,W$ subsets in $\r$.
\begin{itemize}
\item[(H1')] 
For $u_0$ as in (H1), $\d(R(u_0),\partial V)>0$, (recall that by assumption (H2) we have the inclusion (\ref{eq:regularity for trace space}), so $u_0\in C^{2+r}_*$ for some $r>0$).
\end{itemize}
Roughly speaking, the hypothesis (H4') means that $F$ induce an elliptic operators only on a region $V$.\\
Note that, since $W^{s+2-2/q}_q\hookrightarrow C^{2+r}_*$ ($C_{s,q}$ denotes the boundedness constant in the embedding), then the set
\begin{equation*}
\D:=\{u\in W^{s+2-2/q}_q\,:\, \,\d(R(u),\partial V) >0\}\,,
\end{equation*}
is open in $W^{s+2-2/q}_q$. Indeed, fix $v\in \D$ and set $\delta :=d(R(v),\partial V)>0$, then for all $u\in B_{s,q}(v,\delta/(2C_{s,q}))$, $\zeta \in \partial V$ and $x\in \mathbb{R}^d$, we have
\begin{equation*}
|D^2u(x)-\zeta|\geq \Big||D^2v(x)-\zeta|-|D^2v(x)-D^2u(x)|\Big|\geq \frac{\delta}{2}\,;
\end{equation*} 
by the arbitrariness of $x$ and $\zeta$, we obtain $\d(R(u),\partial V)\geq \delta/2$. Then $B_{s,q}(v,\delta/(2C_{s,q}))\subset\D$ and the claim follows.\\
The following proposition allows us to extend the treatment just proposed under the weaker hypothesis (H4').
\begin{proposition}
Let the hypotheses (H2)-(H3)-(H4') be satisfied and let $n$ be an integer. Then there exists a $\lambda_n\in \mathbb{R}$ such that $\lambda_n + M_F(u)\in \mathcal{H}^{\infty}(H^s_q)$ with $\phi_{\lambda_n + M_F(u)}\leq \phi_n < \pi/2$, for all $$u\in \D_n:=\{u\in W^{s+2-2/q}\,:\,d(R(u),\partial V)>1/n\}\cap B_{s,q}(0,n)\,.$$
Furthermore, the inequality in (\ref{eq:prop M uniform of H infinite}) holds for all $u\in \D_n$.
\end{proposition}
\begin{proof}
The proof is the same as the proof of Proposition \ref{prop:Hinfinite M} done in Subsection \ref{ss:proof of Propositon and P1}. One has only to observe that
\begin{equation*}
\Psi_k(D)D^2u + t\psi_{k+1}(D)D^2 u\rightarrow_{k\rightarrow\infty} D^2 u\,, \qquad \text{in}\,\,\, L_{\infty}(\r)\,,
\end{equation*}
uniformly in $t\in [0,1]$.
\end{proof}
Now we are ready to prove the existence of a maximal defined solution for the system (\ref{EQ}) under the weaker hypothesis (H4'), note that the blow up criterion (\ref{eq:lim extended}) change its form.
\begin{theorem}
\label{t:fully on a portion}
Let the hypotheses (H1')-(H2)-(H3)-(H4') be satisfied. Then the fully nonlinear parabolic problem (\ref{EQ}) has a unique maximal defined solution of class in 
\begin{equation*}
H^1_{q,loc}([0,T);H^s_q)\cap L_{q,loc}([0,T);H^{2+s}_q)\cap C([0,T);\D)\,;
\end{equation*}
where $T=T(u_0)$. Furthermore, one of the following are satisfied
\begin{itemize}
\item[$i)$] $T(u_0)=\infty$.
\item[$ii)$] $T(u_0)<\infty$ and
\begin{equation*}
\lim_{t\nearrow T(u_0)} u(t)\,\,\,\text{does not exist in}\,\, W^{2+s-\frac{2}{q}}_{q}\,.
\end{equation*}
\item[$iii)$] $T(u_0)<\infty$ and
\begin{equation*}
\d_{\gamma}(u(t),\partial \D)\searrow 0\,,
\end{equation*}
as $t\nearrow T(u_0)$, (here $\d_{\gamma}(u(t),\partial \D):=\inf\{|u(t)-v|_{W^{s+2-2/q}_q}\,:\,v\in \D\}$).
\end{itemize}
\begin{proof}
The proof is similar to proof of Theorem \ref{t:well posedness} and it consists in an easy adaptation of Corollary 5.1.2 in \cite{Pruss}.
\end{proof}

\end{theorem}
\subsection{Parabolic Stochastic PDEs and Proof of Theorem \ref{t:wellposedness stochastic}}
\label{ss:stochastic proof}
In this section we provide the proof of Corollary \ref{t:wellposedness stochastic}, this will be an easy consequence of a more general result, i.e. Theorem \ref{t:wellposedness stochastic refined}.\\
Throughout this section, $(\Omega,\A,\P)$ denotes a probability space, endowed with a filtration $\f=\{\f_t\}_{t\in \R^+}$ which satisfies the usual conditions.\\
In the context of stochastic parabolic fully nonlinear partial differential equations, we can admits that the nonlinearities $F(x,D^2 u)$ depends on $\omega\in \Omega$. For the sake of completeness, below we list our hypotheses:
\begin{itemize}
\item[(S1)] $s>0$, $2\leq q<\infty$ and $q>(2+d)/s$.
\item[(S2)] For each $\omega\in \Omega$, the map $(x,\zeta)\mapsto  F(\omega,x,\zeta)$ is a smooth function of its arguments, its derivatives are uniformly bounded in $x\in \R^d$ and $\omega\in \Omega$.\\ 
Furthermore, $|F(\cdot,0)|_{L_{\infty}(\Omega;H^s_q)}<\infty$ (here $s,q$ are as in (S1)).
\item[(S3)](Strongly parabolicity hypothesis). For each $n\in \mathbb{N}$ there exists $c'_n>0$ such that
\begin{equation*}
\sum_{|\alpha|=2}\frac{\partial F}{\partial \zeta_{\alpha}}(\omega,x,\zeta)\, \xi^{\alpha} \geq c'_n |\xi|^2\,,\qquad \forall \xi
\in \mathbb{R}^d\,,
\end{equation*}
uniformly in $\omega\in \Omega\,,\,x\in \mathbb{R}^d$ and $|\zeta|\leq n$.
\item[(S4)] For each $(x,\zeta)$ and each $|\alpha|\leq 2$, the map 
$$\omega \mapsto \frac{\partial{F}}{\partial \zeta_{\alpha}}(\omega,x,\zeta)$$ 
is $\f_0$-measurable. 
\item[(S5)] $u_0:\Omega \rightarrow W^{s+2-2/q}_q$ is strongly $\f_0$-measurable.
\end{itemize}
Under the hypotheses (S1)-(S4) it is clear that for each $\omega\in \Omega$ we can construct the paradifferential operator as done in Section \ref{s:paradifferential} regarding $\omega$ as a fixed parameter.\\
To be precise, for any $u\in C^{2+r}_*$ we define
\begin{align*}
m_{k}^{u,\alpha}(\omega,x)&=\int_0^1 dt \frac{\partial F}{\partial \zeta_{\alpha}}(\omega,D^2\Psi_k(D)u+tD^2\psi_{k+1}(D)u)\,,\\
M^{\alpha}_F(\omega,u;x,\xi)&=\sum_{k=0}^{\infty}m_{k}^{u,\alpha}(\omega,x) \psi_{k+1}(\xi)\xi^{\alpha}\,;
\end{align*}
similarly one define $M_F(\omega,u;x,\xi)$ and the realization of the paradifferential operator $M_F(\omega,\cdot)$. Before proceeding further, we prove the following measurability result.
\begin{lemma}
\label{l:measurability}
Under the hypotheses (S1)-(S4), the following holds:
\begin{itemize}
\item[i)] For each $u\in W^{s+2-2/	q}_q$ and $v\in H^{s+2}_q$,
\begin{align*}
M_F(u)v:\Omega &\rightarrow H^s_q\,,\\ 
\omega&\mapsto M_F(\omega,u)v\,,
\end{align*}
is $\f_0$-strongly measurable.
\item[ii)] For each $u\in W^{s+2-2/	q}_q$, the map
\begin{align*}
G(\cdot,u):\Omega &\rightarrow H^s_q\,,\\ 
\omega&\mapsto-F(\omega,x,D^2\Psi_0(D)u)\,,
\end{align*}
is $\f_0$-strongly measurable.
\end{itemize}
\end{lemma}
\begin{proof} 
We recall some basic facts which we will use freely in the proof of the Lemma. For each $r\in(1,\infty)$ and $p\in(1,\infty)$, we have there is a natural identification $(H^s_p)^*=H^{-s}_{p'}$ (where $1/p+1/p'=1$) and the Schwartz class $\S(\R^d)$ is dense in $H^{r}_{p}$ (see \cite{Interpolation}). In particular, $\S(\R^d)\subset (H^s_p)^*$ is weak*-dense.\\
In the following $\<\cdot,\cdot\>$ denotes the pairing in the duality between $H^s_p$ and $H^{-s}_{p'}$.\\
Lastly, we recall that the pointwise convergence preserve measurability.\\
i) By Pettis measurability Theorem (see \cite{Analysis1}) it is enough to show that, for each $f\in \S(\R^d)$, the maps $\omega \mapsto \<M_F(\omega,u)v,f\>$ is $\f_0$-measurable.\\
We first prove the claim under the additional hypothesis $v\in \S(\R^d)$. For such $v$, 
\begin{multline}
\label{eq:riemann integral}
\<M_F(\omega,u)v,f\> =\int_{\R^d} \int_{\R^d} e^{2\pi i \xi \cdot x} M_F(\omega,u;x,\xi)\F(v)(\xi)f(x)\,d\xi\,dx\\
=\sum_{|\alpha|\leq 2}\int_{\R^d} \int_{\R^d}e^{2\pi i \xi \cdot x} \left(\sum_{k=0}^{\infty} m^{u,\alpha}_k(\omega,x)\psi_{k+1}(\xi)\xi^{\alpha}\right)\F(v)(\xi)f(x)d\xi\,dx\,,
\end{multline}
where, as before,
\begin{equation*}
m^{u,\alpha}_k(x,\omega)=
\int_0^1\frac{\partial F}{\partial \zeta_{\alpha}}(\omega,\Psi_k(D)D^2u + t\psi_{k+1}(D)u)\,dt\,.
\end{equation*}
Recall that, by hypothesis (S2), for all $\omega\in\Omega$ the map $(x,\zeta)\mapsto F(\omega,x,\zeta)$, then the last integral in (\ref{eq:riemann integral}) converges as a Riemann integral. This implies that $\omega\mapsto\<M_F(\omega,u)v,f\>$ is $\f_0$-measurable.\\
If $v\in H^{s+2}_q$, then choose a sequence of $\{v_n\}_n\subset \S(\R^d)$, then
\begin{equation*}
\<M_F(\omega,u)v,f\> =\lim_{n\nearrow \infty} \<M_F(\omega,u)v_n,f\>\,, \qquad \forall \omega\in \Omega\,,
\end{equation*}
since $M_F(\omega,u)\in \B(H^{s+2}_q,H^s_q)$. Then the claim follows.\\
\\
ii). By the same argument, it is enough to show that for each $f\in \S(\R^d)$, then the map $\omega \mapsto \<G(\omega,u),f\>$ is $\f_0$-measurable. Since,
\begin{equation*}
 \<G(\omega,u),f\>=\int_{\R^d} - f(x)\,F(\omega,x,D^2\Psi_0(D)u)\,dx\,,
\end{equation*}
the claim follows as in i).
\end{proof}
Before stating our main result, we recall some basic notation and definitions.
\begin{definition}[$\f$-cylindrical Brownian motion]
Let $\H$ be an Hilbert space. A bounded linear operator $W_H:L_2(\R^+;\H)\rightarrow L_2(\Omega)$ is called an $\f$-cylindrical Brownian motion, if the following are satisfied:
\begin{itemize}
\item[i)] For all $f\in L_2(\R^+;\H)$, then $W_H(f)$ is a centred Gaussian random variable.
\item[ii)] For all $t\in \R^+$ and $f\in L_2(\R^+;\H)$ with support in $[0,t]$, $W_H(f)$ is $\f_t$-measurable.
\item[iii)] For all $f_1,f_2\in L_2(\R^+;\H)$, then $\E(W_H(f_1)W_H(f_2))=[f_1,f_2]_{L_2(\R^+;\H)}$. 
\end{itemize} 
\end{definition}
\begin{example}
\label{ex:1 dimensional Brownian motion}
(One dimensional Brownian motion). Let $\{\beta_t\}_{t\in \R^+}$ be a standard Brownian motion on $\Omega$, it can be viewed as a $\f^b$-cylindrical Brownian motion ($\f^b$ has already been defined in Section \ref{s:introduction}). Indeed, we identify $\{\beta_t\}_{t\in \R^+}$ with
\begin{equation*}
W_{\R}^b(\1_{[0,t]}):=\beta_t\,, \qquad t\in \R^+\,.
\end{equation*}  
Note that, $W_{\R}^b$ is uniquely identified by the previous formula.
\end{example}
\begin{example}
(Space-time white noise). Any space-time white noise $W$ on $\R^d$ defines a cylindrical Brownian motion on $L_2(\R^d)$ and vice versa by the formula:
\begin{equation*}
W_{L_2(\R^d)}(\1_{[0,t]}\otimes \1_{B})=W(t,B)\,,
\end{equation*}
where $t\in \R^+$ and $B\subset \R^d$ is a Borel set of finite measure.
\end{example}
Further examples can be found in \cite{QuasilinearStochastic,InternetSeminar,StochasticEvolution}.\\
\\
In the following definition, $\{\gamma_n\}_{n\in \mathbb{N}}$ is a sequence of independent standard Gaussian random variable on some probability space $(\Omega',\A',\P')$.
\begin{definition}[$\gamma$-radonifying operators]
\label{def:gamma radonifying}
As before $\H$ is an Hilbert space and let $X$ be a reflexive Banach space. Then a bounded linear operator $T\in \B(\H,X)$ is said to be $\gamma$-radonifying (or briefly $T\in \gamma(\H,X)$) if
\begin{equation*}
\sup \E' \Big|\sum_{k=1}^n \gamma_k Th_k\Big|^2_X<\infty\,,
\end{equation*}
where the supremum is taken over all finite orthonormal systems $\{h_k\}_{k=1}^n$ in $\H$.
\end{definition}
The hypothesis of reflexivity in Definition \ref{def:gamma radonifying} is not necessarily for defining $\gamma$-radonifying operators; we will not need this here and we refer to \cite{Analysis2} for more on this topic.\\
\\
To prove Theorem \ref{t:wellposedness stochastic} we first analyse the well posedness of the following abstract quasilinear evolution equation:
\begin{equation}
\label{eq:stochastic abstract}
\begin{cases}
du+M_F(u)u\,dt = (G(u)+K(t,u)+g)dt + (B(t,u)+b)dW_{\H}\,,\\
u(0)=u_0\,.
\end{cases}
\end{equation}
for $t\in[0,T]$. Where:
\begin{itemize}
\item[(S6)] $K:\Omega \times [0,T]\times W^{s+2-2/q}_q \rightarrow H^s_q$ is strongly measurable and the map $\omega\mapsto K(\omega,t,x)$ is for all $t\in [0,T]$ and $x\in H^{s+2}_q$ strongly $\f_t$-measurable.
Moreover, for all $n\in \N$ there exists $L_n^K>0$, such that $K$ is locally Lipschitz continuous, i.e.
\begin{equation*}
|K(\omega,t,x)-K(\omega,t,y)|_{H^s_q}\leq L_n^K |x-y|_{W^{s+2-2/q}_q}\,.
\end{equation*}
\item[(S7)] $B:\Omega \times [0,T]\times W^{s+2-2/q}_q \rightarrow \gamma(\H,H^{s+1}_q)$ is strongly measurable and the map $\omega\mapsto B(\omega,t,x)$ is for all $t\in [0,T]$ and $x\in H^{s+2}_q$ strongly $\f_t$-measurable.
Moreover, for all $n\in \N$ there exists $L_n^B>0$, such that $K$ is locally Lipschitz continuous, i.e.
\begin{equation*}
|B(\omega,t,x)-B(\omega,t,y)|_{\gamma(\H,H^{s+1}_q)}\leq L_n^B |x-y|_{W^{s+2-2/q}_q}\,.
\end{equation*}
\item[(S8)] The functions $f:\Omega \times [0,T]\rightarrow H^s_q$ and $b:\Omega \times[0,T]\rightarrow \gamma(\H,H^{s+1}_q)$ are strongly measurable and adapted to $\f$. Moreover, 
$$
f\in L^q(\Omega\times [0,T];H^s_q)\,, \quad b\in L^q(\Omega \times [0,T];\gamma(\H,H^{s+1}_q))\,.
$$
\end{itemize}
For the problem (\ref{eq:stochastic abstract}) we have the following notion of maximal defined solution, which is a adaptation in our situation of Definition 4.1-4.2 in \cite{QuasilinearStochastic}; for the definition of stochastic integrability see \cite{StochasticIntegration}.
\begin{definition}[Maximal local solution for (\ref{eq:stochastic abstract})]
\label{def:solution stochastic}
Let $n\in \N$ and let $\sigma,\sigma_n,$ be $\f$-stopping times with $0\leq \sigma,\sigma_n\leq T$ almost surely.\\ 
Let $u:\Omega \times [0,\sigma)\rightarrow H^s_q$ (here $\Omega\times [0,\sigma):=\{(\omega,t)\in \Omega\times [0,T]\,:\,0\leq t<\sigma(\omega)\}$) strongly measurable and adapted.
\begin{itemize}
\item[i)] We say that $(u,\{\sigma_n\}_n,\sigma)$ is a \textit{local solution} of (\ref{eq:stochastic abstract}), if $\{\sigma_n\}_n$ is an increasing sequence with $\lim_{n\rightarrow\infty}\sigma_n=\sigma$ pointwise almost surely and for all $n\in \N$ we have
\begin{equation*}
u(\omega,\cdot)\in L_q(0,\sigma(\omega);H^{s+2}_q)\cap C([0,\sigma_n(\omega)];W_q^{s+2-2/q})\,,
\end{equation*} 
for almost all $\omega\in \Omega$. Moreover, for each $n\in \N$, $\1_{[0,\sigma_n]}B(u)$ is stochastically integrable and the identity
\begin{multline*}
u(t)-u_0 +\int_{0}^t M_F(u(s))u(s)\,ds\\ 
= \int_{0}^t G(u(s))+K(s,u(s))+g(s))\,ds + \int_0^t( B(u(s))+b(s))\,dW_{\H}(s)\,,
\end{multline*}
holds for almost all $\omega\in \Omega$ and all $t\in [0,\sigma_n(\omega)]$.\\
Furthermore, we say that the local solution is \textit{unique}, if for every local solution $(v,\{\tau_n\}_n,\tau)$ satisfies $u(\omega,t)=v(\omega,t)$ for all $\omega\in \Omega$ and $t\in [0,\min\{\sigma(\omega),\tau(\omega)\})$.
\item[ii)] We say that $(u,\sigma_n,\sigma)$ is a \textit{maximal unique local solution}, if for any other local solution $(v,\{\tau_n\}_n,\tau)$, we have almost surely $\tau\leq \sigma$ and $u(\omega,t)=v(\omega,t)$ for all $\omega\in \Omega$ and $t\in [0,\sigma(\omega))$.
\end{itemize}
\end{definition}
Let $I\subset \R$ an interval, then $BUC(I;X)$ means the space of all bounded uniformly continuous function with value in $X$. We are now in position to prove the main result of this section:
\begin{theorem}
\label{t:wellposedness stochastic refined}
Under the hypotheses (S1)-(S8), the problem (\ref{eq:stochastic abstract}) has a maximal unique local solution $(\sigma,\{\sigma_n\}_n,u)$, $($see Definition $(\ref{def:solution stochastic}))$. Moreover, the following blow-up criterion holds:
\begin{equation}
\label{eq:stochasticblow-up criterion}
\P\{\sigma<T\,,\,|u|_{L_q(0,\sigma;H^{s+2}_q)}<\infty\,,\,u\in BUC([0,\sigma);W_q^{s+2-2/q})\}=0\,.
\end{equation}
\end{theorem}
Before starting the proof, recall that
\begin{equation*}
[H^s_q,H^{s+2}_q]_{1/2}=H^{s+1}_q\,, \quad (H^s_q,H^{s+2}_q)_{1-q/q,q}=W^{s+2-2/q}_q\,;
\end{equation*}
see for instance \cite{Interpolation,Analysis1}.
\begin{proof}[Proof of Theorem \ref{t:wellposedness stochastic refined}]
The proof consists in an application of Theorem 4.11 in \cite{QuasilinearStochastic}; we are left to verify the conditions [Q1]-[Q3], [Q4*]-[Q7*], [Q8]-[Q9] in Section 4 of \cite{QuasilinearStochastic}.\\
Since $H^s_q$ is isomorphic to $L_q$ and $q\geq 2$ by (S1) (as noted before the pseudodifferential operator $\Lambda^s$ gives such isomorphism), then the condition [Q1] is satisfied.\\
The condition [Q2] is indeed hypothesis (S5), instead [Q3] follows from Lemma \ref{l:measurability}.\\
Conditions [Q4*] and [Q5*] are implied respectively by Proposition \ref{prop:Hinfinite M} and property (P1); moreover [Q7*] follows from property (P2) and hypothesis (S6).\\ 
Lastly, [Q7*] follows by (S7), [Q8] is automatically verified by our assumption and [Q9] is indeed hypothesis (S8). 
\end{proof}
It is clear that, Corollary \ref{t:wellposedness stochastic} is a trivial consequence of Theorem \ref{t:wellposedness stochastic refined}; we omit the details.\\
\\
We conclude this section with some comments on the lower order nonlinearities $K,B$ appering in (\ref{eq:stochastic abstract}). Indeed, the hypotheses on this terms in \cite{QuasilinearStochastic} ([Q6*]-[Q7*] in Subsection 4.2) are weaker then the our. We choose to use the stronger assumptions (S6)-(S7) for two reason. Firstly, our focus is on the fully nonlinearity $F(x,D^2u)$ rather than "lower order" terms; secondly if we use the weaker assumptions in \cite{QuasilinearStochastic} then it forces us to explain other notions which can be misleading for the reader.\\
The interested reader can easily relax the hypothesis (S6)-(S7), thanks to the result of abstract quasilinear parabolic evolution equations in \cite{QuasilinearStochastic}.
\section{Extensions}
\label{s:extensions}
\subsection{High Order and non-Autonomous fully Nonlinear (S)PDEs}
\subsubsection{High Order Fully Nonlinear (S)PDEs}
The approach followed in Section \ref{s:proof of the main results}, to prove existence of a solution, can be extended to high order fully nonlinear parabolic PDEs. For instance, one could replace $F(x,D^2 u)$ by $F(x,D^m u)$ in (\ref{EQ}), (\ref{eq:stochastic intro}) (where $m\in 2\mathbb{N}$) and the strongly parabolicity hypothesis (H4) by:
\begin{itemize}
\item For all $n\in \N$, there exists $c'_n>0$ such that
\begin{equation*}
\sum_{|\alpha|=m}\frac{\partial F}{\partial \zeta_{\alpha}}(x,\zeta)\,\xi^{\alpha} \geq c'_n |\xi|^m\,,\qquad \forall \xi\in \mathbb{R}^d\,;
\end{equation*}
for all $x\in \R^d$ and $|\zeta|<n$.
\end{itemize}
With clear adaptation of hypothesis (H4') in Subsection \ref{ss:deterministic proof} and (S3)-(S4) in Subsection \ref{ss:stochastic proof}.\\
Then, Theorems \ref{t:well posedness} and \ref{t:wellposedness stochastic refined} still holds and for the adapted version of the respectively hypotheses.\\
Indeed, the construction of the paradifferential operator also holds in this case (see \cite{Taylor} for more details) although one has to choose $q$ large enough to make valid the embedding
$$
(H^{s+m}_q,H^{s}_q)_{1-1/q,q}\hookrightarrow C^{m+r}_*\,,
$$
for some $r>0$. Moreover, the analysis of Section \ref{s:proof of the main results} can be carried over in this case.\\
Finally, it is worthwhile to note that $F(x,D^m u)$ may depend on $\omega\in\Omega$ as in hypothesis (S4) in Subsection \ref{ss:stochastic proof} and Lemma \ref{l:measurability} still holds. We omit the details.
\subsubsection{Non autonomous Fully Nonlinear (S)PDEs}
The treatment developed in Sections \ref{s:H infinite bounded}-\ref{s:paradifferential} can be also extended to the non autonomous cases. Indeed, for $m\in 2\N$, one can replace the decomposition (\ref{eq:intro decomposition}) by
\begin{equation*}
F(t,x,D^m u)= M_F(u;t,x,D)u + F(t,x,D^m\Psi_0(D)u)\,.
\end{equation*}
Here $M_{F}(u;t,x,\xi)$ is defined as in Section \ref{s:paradifferential} considering $t$ as a parameter.\\
Denoting with $M_{F}(u,t)$ the realization of $M_{F}(u;t,x,D)$ on $H^s_q$, the problem (\ref{EQ}) with a time-depending $F$ can be rewritten as
\begin{equation}
\label{eq:non autonomous}
\begin{cases}
u' + M_{F}(u,t)u = G(u,t)\,, \quad\,\,t>0\,,\\
\qquad u(0)=u_0\,;
\end{cases}
\end{equation}
where $G(u,t):=-F(t,x,D^2\Psi_0(D)u)$ and similarly for the stochastic case.\\
The treatment of non autonomous equations as in (\ref{eq:non autonomous}), is not as known as the autonomous case, so we limit ourself to autonomous case, although the fully nonlinear parabolic (stochastic) problem with time dependent $F$ can be analysed as soon as one has got results on quasilinear non-autonomous abstract parabolic evolution equations.
\section{Comments}
\label{s:comments}
We now move to compare our approach to other known results. A quite amount of work with fully nonlinear parabolic partial differential equations is done in \cite{Lunardi} (see also \cite{DaPrato-maximal}); but the approach taken here is completely different and the results appear not in the form of Theorem \ref{t:well posedness}. Indeed, some application of the theory are more suited for H\"{o}lder regularity rather than Sobolev regularity.\\ 
We mention that the results of \cite{Lunardi} are suited to deal with fully nonlinear equations on domains with boundary, instead our approach does not seem to be so flexible.\\
Our approach is more similar to the one deviced in \cite{Taylor} Chapter 15 Section 8. Although there a paradifferential technique is used, the way to produce local existence is completely different. Indeed, the the local existence is proven by a Gal\"{e}rkin method and a compactness argument.\\
\\
In the context of stochastic partial differential equations, our method to prove the existence of a maximal defined solution (to our knowledge) appears to be new. Moreover, it permits us to consider a very general noise (see Subsection \ref{ss:stochastic proof}), instead in \cite{DaPrato,P.L.Lions,P.L.Lions2} the driving process is an $m$-dimensional Brownian motion.\\
\\
Lastly, the approach taken here seems to be suitable for studying fully nonlinear parabolic (S)PDEs on a closed Riemannian manifold. Indeed, miming the localization technique used Chapter 6 Section 6.5 of \cite{Pruss} one can reduce the proof of the existence of a fully nonlinear (S)PDEs on a closed manifolds to an equation of the form (\ref{eq:stochastic intro}) or (\ref{eq:stochastic abstract}) with a vector valued $F$. This would be a very interesting distinguish fact of our approach, since to our knowledge this is not already studied.

\end{document}